\documentclass[reqno,a4paper]{amsart}
\usepackage[utf8]{inputenc}
\usepackage[T1]{fontenc}
\usepackage[british]{babel}
\usepackage{amsfonts, amssymb, amsmath}
\usepackage[all]{xy}
\SelectTips{cm}{}
\usepackage{hyperref}
\usepackage{mathbbol}
\usepackage{calrsfs}
\usepackage{mathabx}
\usepackage{amsfonts}
\usepackage{amsthm}
\usepackage{graphicx}
\allowdisplaybreaks
\usepackage{xcolor}
\usepackage{enumerate}

\theoremstyle{plain}
\newtheorem{theorem}{Theorem}[section]
\newtheorem{proposition}[theorem]{Proposition}
\newtheorem{corollary}[theorem]{Corollary}

\theoremstyle{definition}
\newtheorem{definition}[theorem]{Definition}

\theoremstyle{remark}
\newtheorem{remark}[theorem]{Remark}
\newtheorem{remarks}[theorem]{Remarks}
\newtheorem{example}[theorem]{Example}
\newtheorem{examples}[theorem]{Examples}

\newenvironment{tfae}
{
    \begin{enumerate}}
        {\end{enumerate}}

\usepackage{tikz-cd}
\tikzset{commutative diagrams/.cd,
epi/.style = {two heads},
mono/.style = {tail},
cokernel/.code={\pgfsetarrowsend{Triangle[open]}},
kernel/.code={\pgfsetarrowsstart{Triangle[open, reversed]}},
inclu/.style = {hook},
iso/.style = {"\sim" sloped},
pb/.style ={"\lrcorner"{anchor=center, pos=0.125}, draw=none}
}

\newcommand{\defn}[1]{\textbf{#1}}
\newcommand{\noproof}{\hfil\qed}

\newcommand{\N}{\mathbb{N}}

\newcommand{\C}{\mathcal{C}}
\newcommand{\E}{\mathcal{E}}
\newcommand{\Power}{\mathcal{P}}
\renewcommand{\S}{\mathcal{S}}
\newcommand{\U}{\mathcal{U}}
\newcommand{\V}{\mathcal{V}}

\renewcommand{\leqq}{\preccurlyeq}
\newcommand{\To}{\Rightarrow}

\newcommand{\NM}{\mathrm{NM}}
\newcommand{\CExt}{\mathrm{CExt}}
\newcommand{\Ext}{\mathrm{Ext}}
\newcommand{\DExt}{\text{$2$-$\Ext$}}
\newcommand{\DCExt}{\text{$2$-$\CExt$}}
\newcommand{\Hom}{\mathrm{Hom}}

\newcommand{\coker}{\mathrm{coker}}

\newcommand{\op}{\mathrm{op}}

\newcommand{\EXT}{\mathsf{Ext}}
\newcommand{\Set}{\mathsf{Set}}
\newcommand{\Ab}{\mathsf{Ab}}

\title[The cohomology objects of a semi-abelian variety are small]{The cohomology objects of a\\ semi-abelian variety are small}

\date{\today}

\author{Sébastien Mattenet}
\author{Tim Van~der Linden}
\author{Raphaël Jungers}

\email[Sébastien Mattenet]{sebastien.mattenet@uclouvain.be}
\email[Tim Van~der Linden]{tim.vanderlinden@uclouvain.be} \email[Tim Van~der Linden]{tim.van.der.linden@vub.be}
\email[Raphaël Jungers]{raphael.jungers@uclouvain.be}

\address[Sébastien Mattenet, Raphaël Jungers]{Institut de Recherche en Math\'ematique Appliquée, Universit\'e catholique de Louvain, avenue Georges Lema\^itre~4 bte~L4.05.01, B--1348 Louvain-la-Neuve, Belgium}
\address[Tim Van der Linden]{Institut de Recherche en Math\'ematique et Physique, Universit\'e catholique de Louvain, che\-min du cyclotron~2 bte~L7.01.02, B--1348 Louvain-la-Neuve, Belgium}
\address[Tim Van der Linden]{Mathematics \& Data Science, Vrije Universiteit Brussel, Pleinlaan 2, B--1050 Brussel, Belgium}

\thanks{The second author is a Senior Research Associate of the Fonds de la Recherche Scientifique--FNRS}

\subjclass[2020]{03E25, 03E30, 18E13, 18G15, 18G50}

\keywords{Yoneda extension; double extension; crossed extension; Schreier extension; cohomology; semi-abelian category; set}

\begin{document}

\begin{abstract}
    A well-known, but often ignored issue in Yoneda-style definitions of cohomology objects via collections of $n$-step extensions (i.e., equivalence classes of exact sequences of a given length~$n$ between two given objects, usually subject to further criteria, and equipped with some algebraic structure) is, whether such a collection of extensions forms a set. We explain that in the context of a semi-abelian variety of algebras, the answer to this question is, essentially, yes: for the collection of all $n$-step extensions between any two objects, a set of representing extensions can be chosen, so that the collection of extensions is ``small'' in the sense that a bijection to a set exists.

    We further consider some variations on this result, involving double extensions and crossed extensions (in the context of a semi-abelian variety), and Schreier extensions (in the category of monoids).
\end{abstract}

\maketitle

\section{Introduction}
Yoneda's classical approach to cohomology in abelian categories~\cite{Yoneda-Exact-Sequences} involves groups whose elements are equivalence classes (here called \defn{$n$-step extensions}) of exact sequences
\[
    0\to K\to X_n\to\cdots\to X_1\to Q\to 0
\]
of a fixed length $n\geq 1$ between two given objects $K$ and $Q$. A priori, it is not clear that these groups are legitimate: in question is, whether they have underlying (small) sets. Indeed, already in the case of one-step extensions, so isomorphism classes of short exact sequences
\[
    0\to K\to X_1\to Q\to 0
\]
where the objects $K$ and $Q$ are fixed, there seems to be no reason in general for the collection of all such to form a set.

(Actually, strictly speaking, this \emph{must} be false, because for any given one-step extension of abelian groups, the short exact sequences it consists of form a proper class, which would be in contradiction with the fact\footnote{We reason in the standard context of (ZFC) with a chosen Grothendieck universe. Details are recalled in Section~\ref{Section Sets} below.} that the elements of a set are always sets themselves. On the other hand, it may be shown---see below---that the collection of one-step extensions is in bijection with a small set, namely a chosen set of representative short exact sequences, so it is ``small'' in this sense.)

In the context of an abelian category, powerful homotopical-algebraic techniques have been developed over the past decades to fully address this question, each at its own level of generality. These include the theories of model categories~\cite{Quillen}, encompassing localisation at a chosen class of morphisms~\cite{MR2102294}, derived categories~\cite{Verdier, MR1453167}, and categories of fractions~\cite{Gabriel-Zisman}.

This article aims to investigate the problem in the non-additive setting of semi-abelian categories~\cite{Janelidze-Marki-Tholen}. One way to address this size issue---both here and in related cohomology theories that classify exact sequences---is to reinterpret the cohomology objects. This can be done, for instance, via derived functors, showing that these objects are isomorphic to well-defined groups with underlying sets. Often, a cohomology group is constructed as a subquotient of a hom-set that naturally inherits a group structure from one of the objects. This is the approach of~\cite{MacLane:Homology,Weibel} in the abelian case; see Example~\ref{Example Central Extensions} and Section~\ref{Section Variation} for concrete non-additive illustrations.

But what if no such interpretation is known? In his book ``Homology'', using slightly different terminology, Mac Lane deals with the problem as follows~\cite[end of Section~IX.1]{MacLane:Homology}:
\begin{quote}
    ``To keep the foundations in order we wish the collection of all subobjects of an object $A$ and the collection of all extensions of $A$ by~$C$ both to be sets and not classes. Hence, for an additive category we assume two additional axioms:

    \emph{Sets of sub- and quotient objects.} For each object $A$ there is a set of morphisms $x$, each monic with range $A$, which contains a representative of every subobject of $A$ and dually, for quotient objects of $A$.

    \emph{Set of extensions.} For each pair of objects $C$, $A$ and each $n\geq 1$ there is a set of $n$-fold exact sequences from $A$ to $C$ containing a representative of every congruence class of such sequences (with ``congruence'' defined as in III.5).

    Both axioms hold in all the relevant examples.''
\end{quote}
``Relevant'' here means of course, those examples treated in Mac Lane's book. Can we show that in the semi-abelian setting, the second axiom always holds, using the same techniques as in the abelian context, perhaps?

Certain successful concrete applications notwithstanding, it appears the context of semi-abelian categories is remarkably resilient against the use of standard techniques of homotopical algebra such as the ones cited above. Is it, for example, not at all clear whether the chain complexes in a semi-abelian category admit a suitable Quillen model structure. As explained in detail in~\cite{CRVdL}, homotopy of chain maps crucially depends on additivity of the category; while it is possible to define chain homotopies in semi-abelian categories, it is currently unknown whether these can be characterised by means of a model structure. And even though simplicial objects in a semi-abelian category do always form a model category~\cite{VdLinden:Simp}, any connection with $n$-step extensions which might be there is still to be understood. Likewise, the formalism of derived categories is intrinsically additive. While neither categories of fractions nor localisations are, it is however not clear precisely how they would apply to the problem at hand. It seems to us that such an approach cannot be based too much on ideas that work for abelian categories, since in our experience this tends to fail. For this reason, we have to resort to an ad-hoc argument. The approach of Dwyer, Hirschhorn, Kan \& Smith  outlined in the book~\cite{MR2102294} suggests that a comparison with the very general language of homotopical categories would be worthwhile. We plan to return to this in future work.

In this article we prove that, if a variety of algebras is semi-abelian, then the collection of all $n$-step extensions between any two objects in it is indeed always small. This includes categories of modules over a ring as additive examples, and the categories of groups, rings, algebras over a ring, all varieties of $\Omega$-groups in the sense of Higgins~\cite{Higgins}, Heyting semilattices and loops, for instance, as non-additive ones. Often, in practice, the collection of \emph{all} extensions between two given objects is then further restricted (e.g., in order to enforce compatibility with a given action).

This result is a simple consequence of Bourn \& Janelidze's characterisation~\cite{Bourn-Janelidze} of semi-abelian varieties by means of operations and identities (here Theorem~\ref{SA characterisation}), which allows us to obtain a bound on the collection of one-step extensions between any two given objects  (Corollary~\ref{Tim retract} and the ensuing Theorem~\ref{set many SES}). For extensions of greater length $n\geq 2$, the result (Theorem~\ref{as many n exact sequences as SES}) follows by a reduction (by means of a standard syzygy argument) to the case $n=1$.

In Section~\ref{Section Variation}, we outline a variation on this: we consider \emph{double (central) extensions} and the closely related \emph{crossed extensions}.

Our method also works outside the semi-abelian context, though any such results need to be checked on a case-by-case basis. In Section~\ref{Monoids}, we treat the example of \emph{Schreier extensions} of monoids. This motivates us to provide an overview of the basic definitions in a context which is as wide as possible: this is the subject of Section~\ref{Extensions Defs}. First, however, we comment on the set-theoretical foundations we shall be adopting.

\section{Set-theoretical preliminaries}\label{Section Sets}
In this article, we work in a standard set-theoretical environment---(ZFC) with a chosen Grothendieck universe~\cite[II.\ Appendice: Univers]{SGA4}---making essential use of the Axiom of Choice. For the sake of completeness, we here recall the very basics of this approach, closely following its presentation in~\cite{HHerrlichGEStrecker1973}.

We assume the axioms of (ZFC), so the Zermelo--Fraenkel axioms together with the Axiom of Choice, but instead of calling the objects that satisfy them \emph{sets}, we call them \defn{conglomerates}. We have, for instance, the empty conglomerate $\varnothing$, the conglomerate of natural numbers $\N$, or the power-conglomerate $\Power(A)$ when $A$ is a conglomerate.

The other Zermelo--Fraenkel axioms are, in essence: two conglomerates are equal if and only if they have the same elements; selecting elements in a given conglomerate $A$ by means of a predicate $P$ determines the sub-conglomerate $\{x\in A\mid P(x)\}$ of $A$; the union of all the elements of a conglomerate forms again a conglomerate; given two conglomerates $A$, $B$, the pair $\{A,B\}$ is a conglomerate; and any collection of conglomerates which is in one-one correspondence with a given conglomerate is again a conglomerate.

Choice means that any epimorphism of conglomerates admits a section: whenever $q\colon X\to Q$ is a surjection, there exists $s\colon Q\to X$ such that $qs=1_Q$.

\begin{definition}
    A \defn{universe} is a conglomerate $\U$ satisfying the following axioms:
    \begin{enumerate}[(U\,1)]
        \item $\N\in \U$;
        \item $A\in \U\To \bigcup A\in \U$;
        \item $A\in \U\To \Power(A)\in \U$;
        \item if $I\in \U$ and $f\colon I\to \U$ is a function, then $f(I)\in \U$;
        \item $a\in A\in \U\To a\in \U$.
    \end{enumerate}
    The elements of the conglomerate $\U$ are called \defn{$\U$-sets}, and the subconglomerates of $\U$ are called \defn{$\U$-classes}.
\end{definition}

The idea is that we can carry out all of the usual operations of set theory on the elements of a universe and still the result will be an element of that universe. In fact, it is easy to see that the $\U$-sets again satisfy the axioms of (ZFC).

We assume that a universe exists and fix a universe $\U$ once and for all. We then call the $\U$-sets simply \defn{sets} and we call the $\U$-classes \defn{classes}. By~(U\,5), any set is a class. A~class is \defn{proper} when it is not a set, and \defn{small} when it is. A~conglomerate is \defn{small} when it admits a bijection to a set, and \defn{proper} when it is not a class.

All of the mathematical objects we consider are now defined in terms of these sets and classes, unless otherwise mentioned. (Sometimes we shall consider proper conglomerates.) In particular, a category is assumed to have classes of objects and arrows, is \defn{locally small} when the arrows between any two objects form a set, and is \defn{small} when the class of arrows is a set. For instance, the locally small category $\Set$ has the universe $\U$ for its class of objects.

\begin{remark}\label{Remark on singletons}
    Note that by (U\,5), the elements of any set are again sets. The elements of a class are sets by definition of a class. So no set and no class can have a proper class as an element. On the other hand, the class of all singletons is proper, because otherwise by (U\,4), the class of all sets would itself be a set.
\end{remark}

\section{Extensions in pointed categories}\label{Extensions Defs}
We start by recalling some fundamental definitions, mainly following the approach of~\cite{PVdL2}, leading to the notion of extension in the context of a category with kernels and cokernels.

A category is \defn{pointed} when it admits a \defn{zero object}, which is an object that is both initial and terminal. In a pointed category, for any two objects $A$ and~$B$ we have the \defn{zero morphism} $A\to 0\to B$ from $A$ to $B$.

The \defn{kernel} $\ker(f)\colon {K(f)\to A}$ of a morphism $f\colon A\to B$ is the equaliser of $f$ and $0$, while the \defn{cokernel} $\coker(f)\colon {B\to Q(f)}$ is their coequaliser. (Here we abuse terminology slightly, since (co)kernels are unique up to isomorphism only. We call \emph{the} (co)kernel any chosen morphism satisfying the universal property.)

A category \defn{with kernels and cokernels} (called a \defn{z-exact} category in~\cite{PVdL2}) is a pointed category in which each morphism admits a kernel and a cokernel.

In a category with kernels and cokernels, a \defn{normal monomorphism} is a morphism $k$ for which a morphism $f$ exists such that $k=\ker(f)$, and a \defn{normal epimorphism} is a morphism $q$ for which a morphism $f$ exists such that $q=\coker(f)$. It is not hard to see that a morphism $k$ is a normal monomorphism if and only if $k=\ker(\coker(k))$, while a morphism $q$ is a normal epimorphism if and only if $q=\coker(\ker(q))$. Hence we may define a \defn{short exact sequence}
\[
    \xymatrix{0 \ar[r] & K \ar[r]^k & X \ar[r]^q & Q \ar[r] & 0}
\]
as a pair $(k,q)$ where, equivalently,
\begin{tfae}
    \item $k=\ker(q)$ and $q=\coker(k)$;
    \item $k$ is a normal monomorphism and $q=\coker(k)$;
    \item $k=\ker(q)$ and $q$ is a normal epimorphism.
\end{tfae}

It is easy to see that if $f=me$ with $m$ a monomorphism, then $\ker(f)=\ker(e)$. Dually, if $e$ is an epimorphism, then $\coker(f)=\coker(m)$. It follows that when a morphism $f\colon A\to B$ admits a factorisation $f=me$ where $e\colon A\to I$ is a normal epimorphism and $m\colon I\to B$ is a normal monomorphism, this factorisation is unique up to a unique isomorphism. Hence in a category with kernels and cokernels, it makes sense to say that a morphism $f$ is \defn{normal} when it admits such a so-called \defn{normal image factorisation}. If a normal morphism $f$ factors as $f=me$ where $m$ is a monomorphism and $e$ is an epimorphism, then $m$ is a normal monomorphism, and $e$ a normal epimorphism.

Given $n\geq 1$, an \defn{exact sequence of length $n$} is a sequence of maps
\[
    0\to K\to X_n\to\cdots\to X_1\to Q\to 0
\]
which may be obtained from $n$ short exact sequences $(m_{i+1},e_i)$ spliced together as in Figure~\ref{Figure Splice}. In other words, all of the morphisms $f_{n+1}$, \dots, $f_1$ are normal, and if $f_i=m_ie_i$ denote their normal image factorisations, then each pair $(m_{i+1},e_i)$ is a short exact sequence. Note that an exact sequence of length $1$ is just a short exact sequence (where $X_1=X$ in the above).
\begin{figure}
    \resizebox{\textwidth}{!}{
    \xymatrix@!0@=3em{
    &&&& I_{n} \ar@/^1.5ex/[rd]^-{m_{n}} &&&&
    \cdots \ar@/^1.5ex/[rd]^-{m_2} \\
    0 \ar[r] &
    K \ar@{=}@/_1.5ex/[rd] \ar[rr]^-{f_{n+1}} &&
    X_{n} \ar@/^1.5ex/[ru]^-{e_{n}} \ar[rr]_-{f_{n}} &&
    X_{n-1} \ar@/_1.5ex/[rd]_-{e_{n-1}} \ar[rr]^-{f_{n-1}} &&
    X_{n-2} \ar[r] \ar@/^1.5ex/[ru] & \cdots \ar[r] &
    X_{1} \ar[rr]^-{f_1} \ar@/_1.5ex/[rd]_-{e_1} &&
    Q \ar[r] & 0\\
    && I_{n+1} \ar@/_1.5ex/[ru]_-{m_{n+1}} &&&&
    I_{n-1} \ar@/_1.5ex/[ru]_-{m_{n-1}} &&&& I_1 \ar@{=}@/_1.5ex/[ru]}}
    \caption{Spicing results in exact sequence of length $n$}\label{Figure Splice}
\end{figure}

A \defn{morphism} of exact sequences of length $n$ is a tuple $(\alpha_n,\dots,\alpha_1)$ such that the diagram
\[
    \xymatrix{
    0 \ar[r] & K \ar@{-->}[d] \ar[r]^-{f_{n+1}} &
    X_{n} \ar[d]^-{\alpha_n} \ar[r]^-{f_{n}} &
    X_{n-1} \ar[d]^-{\alpha_{n-1}} \ar[r] & \cdots \ar[r] &
    X_{2}\ar[r]^-{f_{2}} \ar[d]^-{\alpha_{2}} &
    X_{1} \ar[r]^-{f_1} \ar[d]^-{\alpha_{1}} &
    Q \ar@{-->}[d] \ar[r] & 0\\
    0 \ar[r] & L  \ar[r]_-{g_{n+1}} &
    Y_{n}  \ar[r]_-{g_{n}} &
    Y_{n-1}  \ar[r] & \cdots \ar[r] &
    Y_{2} \ar[r]_-{g_{2}} &
    Y_{1} \ar[r]_-{g_1}  &
    R \ar[r] & 0}
\]
commutes, in which the dashed arrows are uniquely induced by taking kernels and cokernels. Identities and composition of morphisms are pointwise. One could define a category of exact sequences of length $n$, but we are only interested in morphisms which keep the objects at the ends fixed, so that the dashed arrows are $1_K$ and $1_Q$, respectively. The resulting category is denoted $\EXT^n(Q,K)$; its objects are called exact sequences of length $n$ \defn{from $K$ to $Q$}.

\begin{remarks}\label{Remark EXT not small}
    \begin{enumerate}
        \item Note this category is, in general, not small. Indeed, even in the case of abelian groups, $\EXT^1(0,0)$ has a proper class of objects, all of which are isomorphic. The reason is that each singleton set $\{X\}$ admits a (unique) abelian group structure, making $0\to \{X\}\to 0$ a short exact sequence, and any two such are isomorphic. Each set $X$ induces a short exact sequence of that kind.
        \item Since the category $\EXT^1(0,0)$ is connected, this further provides an example of a situation where the connected components of a category $\EXT^n(Q,K)$ are proper classes.
    \end{enumerate}
\end{remarks}

The connected components of a category $\EXT^n(Q,K)$ are called \defn{$n$-step extensions (under $K$ and over $Q$)}. These form a conglomerate denoted
\[
    \Ext^n(Q,K)\coloneq\pi_0(\EXT^n(Q,K))\text{.}
\]
Two exact sequences of length $n$ from $K$ to $Q$ represent the same $n$-step extension if and only if there exists a zigzag of morphisms between them. Since, as explained above in the case of abelian groups (Remarks~\ref{Remark EXT not small}), it is likely that the elements of $\Ext^n(Q,K)$ are proper classes, this conglomerate is seldom a set or even a class (Remark~\ref{Remark on singletons}). However, such a proper conglomerate may still be ``small'' in the sense that a bijection to a set exists; more precisely, we may often choose a set of representing short exact sequences. For instance, in the above example, $\Ext^1(0,0)$ is a small conglomerate, since it is in bijection with any singleton set.

As explained in the Introduction, the aim of this article is to provide sufficient conditions on the surrounding category for these collections to be small. Even though this question makes sense in any category with kernels and cokernels, so far we can only provide a general answer in the context of semi-abelian varieties. The next section recalls what these are, and treats the case of one-step extensions. We do not have a comprehensive understanding of when $\Ext^1(Q,K)$ is small outside the context of semi-abelian varieties, but can still say a few things when the problem is considered for monoids, provided the classes are restricted to so-called \emph{Schreier extensions}---see Section~\ref{Monoids}.

\section{One-step extensions in semi-abelian varieties}
By definition, a category is \defn{semi-abelian} in the sense of Janelidze, Márki \& Tholen~\cite{Janelidze-Marki-Tholen} when it is pointed, Barr exact~\cite{Barr-Grillet-vanOsdol} and Bourn protomodular~\cite{Bourn1991} with finite coproducts. Since it may be shown that a semi-abelian category is always finitely (co)complete, kernels and cokernels always exist, so that the approach to $n$-step extensions of the previous section applies.

Let us recall that \defn{Barr exactness} means that finite limits exist, as well as coequalisers of kernel pairs, that regular epimorphisms are pullback-stable, and that each internal equivalence relation is a kernel pair. All abelian categories are Barr exact, as is any (elementary) topos, and any variety of algebras in the sense of universal algebra. The \defn{protomodularity} condition says that the \emph{Split Short Five Lemma} holds, or equivalently, that when given a split epimorphism $q$ with section~$s$ and kernel $k$ as in
\begin{equation*}\label{Split Extension}
    \xymatrix{K \ar[r]^-k & X \ar@<.5ex>[r]^-{q} & Q\text{,} \ar@<.5ex>[l]^-{s}}
\end{equation*}
the middle object $X$ is ``covered'' or ``generated'' by the outer objects $K$ and $Q$ in the sense that the monomorphisms $k$ and $s$ are \emph{jointly extremally epimorphic}, which means that they do not both factor through a proper subobject of~$X$. Further details are given in~\cite{Janelidze-Marki-Tholen,Borceux-Bourn} and \cite{PVdL2}, for instance.

\begin{examples}
    Examples include any abelian category, the category of groups, and more generally any variety of $\Omega$-groups~\cite{Higgins}, of which by definition the signature admits a group operation and a unique constant, which means that each algebra contains the one-element algebra as a subalgebra---such as any variety of Lie algebras over a ring or, more generally, any variety of non-associative algebras over a ring. Examples of a different kind include the dual of the category of pointed sets, the category of loops~\cite{Borceux-Clementino,EverVdL4}, and the category of cocommutative Hopf algebras over a field~\cite{MR3958087}.
\end{examples}

In the case of a variety of universal algebras, Bourn \& Janelidze found the following characterisation of semi-abelianness in terms of operations and identities~\cite{Bourn-Janelidze}:

\begin{theorem}\label{SA characterisation}
    Let $\V$ be a variety of algebras. $\V$ is semi-abelian if and only if the free algebra over the empty set is a singleton (in other words, there is a unique nullary operation, a constant denoted $0$), and there exist
    \begin{itemize}
        \item an integer $\ell\geq 1$,
        \item $\ell$ binary operations $\alpha_i$ such that $\alpha_i(x,x)=0$ for all $i=1,\dots, \ell$, and
        \item an $(\ell+1)$-ary operation $\beta$ such that $\beta(\alpha_1(x,y),\dots,\alpha_\ell(x,y),y)=x$.\noproof
    \end{itemize}
\end{theorem}

\begin{remark}\label{Remark zeroes}
    These identities imply the identity $\beta(0,\dots,0,y)=y$, which will be used in what follows. Indeed, $\beta(0,\dots,0,y)=\beta(\alpha_1(y,y),\dots,\alpha_\ell(y,y),y)=y$.
\end{remark}

\begin{examples}
    \begin{enumerate}
        \item For a variety of $\Omega$-groups, we may choose $\ell=1$ and use the group operation to define $\alpha_1(x,y)=xy^{-1}$ and $\beta(z,t)=zt$.
        \item The variety of loops is semi-abelian, since we may put $\alpha_1(x,y)=x/y$ and $\beta(z,t)=z\cdot t$; see~\cite{Borceux-Clementino,EverVdL4} for details.
        \item    The semi-abelian variety of Heyting semilattices is special~\cite{Johnstone:Heyting}, since necessarily here, $\ell\geq2$.
    \end{enumerate}
\end{examples}

Turning to extensions in semi-abelian categories, a first important result we should recall is the Short Five Lemma, whose validity implies that the equivalence classes the conglomerate $\Ext^1(Q,K)$ consists of, are isomorphism classes:

\begin{theorem}[Short Five Lemma~\cite{Bourn1991}]
    In a semi-abelian category, consider a morphism of short exact sequences.
    \[
        \begin{tikzcd}
            0 \ar[r] & K \ar[r,"k"] \ar[d,"\kappa"'] & X \ar[r,"q"] \ar[d,"\xi"] & Q \ar[d,"\eta"] \ar[r] & 0 \\
            0 \ar[r] & {K'} \ar[r,"k'"'] & {X'}\ar[r,"q'"'] & {Q'} \ar[r] & 0
        \end{tikzcd}
    \]
    If $\kappa$ and $\eta$ are isomorphisms, then so is $\xi$.\noproof
\end{theorem}

So in this context, for a given one-step extension over $K$ and under $Q$, the middle objects in its representing short exact sequences are all isomorphic. Those middle objects ``extend'' the kernel object.

Next, we restrict ourselves to the varietal setting, on the way to proving the key Theorem~\ref{set many SES}. This crucially depends on the next result, borrowed from~\cite{PVdL2}. It is a strengthening of Proposition~3.3 in \cite{MMClementinoAMontoliLSousa2015}, which in turn was based on work of Inyangala~\cite{Inyangala} and Gray \& Martins-Ferreira~\cite{JRAGrayNMartinsFerreira2018}.

\begin{corollary}\label{Tim retract}
    Let $\V$ be a semi-abelian variety of algebras and $\ell$ an integer as in Theorem~\ref{SA characterisation}. Then any short exact sequence
    \[
        \xymatrix{0 \ar[r] & K \ar[r]^k & X \ar[r]^q & Q \ar[r] & 0}
    \]
    is a retract over $Q$ of the short exact sequence
    \[
        \xymatrix@C=3em{0 \ar[r] & K^{\ell} \ar[r]^-{(1_{K^{\ell}},0)} & K^{\ell}\times Q \ar[r]^-{\pi_Q} & Q \ar[r] & 0}
    \]
    in the category of pointed sets.
\end{corollary}
\begin{proof}
    Without any loss of generality we may treat the kernel $k$ as a subalgebra inclusion and thus view the elements of $K$ as special elements of $X$. Choose a section $s\colon Q\to X$ of $q$, and define the functions
    \begin{align*}
        \phi \colon K^\ell\times Q\to X & \colon (k_1,\dots,k_\ell,y)\mapsto \beta(k_1,\dots,k_\ell,s(y))     \\
        \psi \colon X\to K^\ell\times Q & \colon x\mapsto (\alpha_1(x,sq(x)),\dots,\alpha_\ell(x,sq(x)),q(x))
    \end{align*}
    Here we use the identities satisfied by the $\alpha_i$. The equality $\phi(\psi(x))=x$ follows immediately from the identity involving $\beta$.

    We further note that $\pi_Q\psi=q$ and
    \begin{align*}
        q(\phi(k_1,\dots,k_\ell,y)) & = q(\beta(k_1,\dots,k_\ell,s(y)))                 = \beta(q(k_1),\dots,q(k_\ell),q(s(y))) \\
                                    & = \beta(0, \dots, 0, y)= y=\pi_Q(k_1,\dots,k_\ell,y)
    \end{align*}
    so that $q\phi=\pi_Q$. This already proves our claim in the category of sets.

    Let us now consider compatibility with the canonical base-points (induced by the constant in the theory of $\V$). It turns out that this depends on the section $s$ preserving $0$. Of course, it is always possible to choose such an $s$.

    For $y\in Q$, we have
    \begin{align*}
        \psi(s(y)) & =(\alpha_1(s(y),sqs(y)), \dots, \alpha_\ell(s(y),sqs(y)),qs(y))                   \\
                   & =(\alpha_1(s(y),s(y)), \dots, \alpha_\ell(s(y),s(y)),y)          =(0, \dots, 0,y)
    \end{align*}
    so that $\psi s=(0,\dots,0,1_Q)$. Likewise, $\phi(0,\dots,0,y) = \beta(0,\dots,0,s(y)) = s(y)$ and so $s=\phi (0,\dots,0,1_Q)$. This proves that the canonical section of $\pi_Q$ is compatible with the section $s$ of $q$. In particular, taking $y=0$ we see that $\phi$ and $\psi$ preserve the base-point as soon as so does $s$.
\end{proof}

\begin{remark}
    Note that since neither the section $s$, nor the number $\ell$ in the characterisation of a semiabelian variety is unique, the algebra $X$ can have several presentations as a retract of $K^\ell\times Q$, one for each~$\ell$ and each chosen section $s$.
\end{remark}

From this we deduce:
\begin{theorem}\label{set many SES}
    For any two algebras $K$, $Q$ in a semi-abelian variety $\V$, the conglomerate $\Ext^1(Q,K)$ is small.
\end{theorem}
\begin{proof}
    Once and for all, we fix a number $\ell$ as in Corollary~\ref{Tim retract}. We denote by $\NM(Q,K,\ell)$ the set of all normal monomorphisms $k\colon K\to X$ in $\V$ where the underlying set of the codomain $X$ is a subset of $K^\ell\times Q$. This conglomerate is indeed a set, because (1) there is only a set of subsets of $K^\ell\times Q$; (2) on any set, there is only a set of $\V$-algebra structures; and (3) between any two algebras there is only a set of morphisms. To each one-step extension under $K$ and over~$Q$, we now associate a representative short exact sequence $(k,q)$ from $K$ to $Q$ whose normal monomorphism $k$ belongs to $\NM(Q,K,\ell)$ and whose normal epimorphism $q$ is a restriction of the projection~$\pi_Q$. Since any such short exact sequence uniquely determines the extension to which it belongs, the conglomerate of extensions is then in bijection with a set.

    We put a well-order $\leqq$ on the set $\NM(Q,K,\ell)$. Given a  one-step extension under~$K$ and over $Q$ now, we consider the minimum for the order $\leqq$ of all short exact sequences $(k,q)$ from $K$ to $Q$ in this extension, where $k$ belongs to $\NM(Q,K,\ell)$ and where the normal epimorphism $q$ is the restriction of $\pi_Q$ to $X$. Provided a short exact sequence of that kind always exists, so that the subset of the well-order $(\NM(Q,K,\ell),\leqq)$ we take a minimum of is non-empty, the extension and the thus chosen representative do indeed determine each other.

    It is here that we apply Corollary~\ref{Tim retract}. Take any element $(k,q)$ of the given extension. Take a base-point--preserving section $s$ of $q$. The corollary provides us with a monomorphism $\psi\colon X\to K^\ell\times Q$ over $Q$ in the category of pointed sets. It determines subset $X'$ of $K^\ell\times Q$ onto which the algebra structure of $X$ may be transported via $\psi$. It is clear that the induced normal epimorphism $q'\colon X'\to Q$ is a restriction of $\pi_Q$. Together with the kernel $k'=\psi k\colon K\to X'$, we find the needed representing short exact sequence $(k',q')$.
\end{proof}

\begin{example}[Central extensions]\label{Example Central Extensions}
    As explained in~\cite{Gran-VdL}, we may consider the sub-conglomerate $\CExt^1(Q,K)$ of $\Ext^1(Q,K)$ consisting of those extensions which are \defn{central}. One way of defining these, is by saying that for any representative short exact sequence $0\to K\to X\to Q\to 0$, the commutator $[K,X]$ must vanish. In particular then, $K$ is an abelian object. In the article~\cite{Gran-VdL}, the conglomerate $\CExt^1(Q,K)$ is implicitly shown to be small, since it admits an interpretation as a cohomology group. Now we may view this result as a simple consequence of Theorem~\ref{set many SES}.
\end{example}

\section{The syzygy argument}\label{Section Syzygy}

Let $\C$ be a category with kernels and cokernels. An object $P$ of $\C$ is \defn{normal-projective} if for every normal epimorphism
\begin{tikzcd}[cramped, sep=small]
    X\ar[r,cokernel]&Y
\end{tikzcd} and for every morphism ${P\to Y}$ there exists a morphism ${P\to X}$ making the diagram
\[
    \begin{tikzcd}
        &X\\
        P&Y
        \arrow["\exists",dashed,from=2-1,to=1-2]
        \arrow["\forall"',from=2-1,to=2-2]
        \arrow[cokernel,from=1-2,to=2-2]
    \end{tikzcd}
\]
commute. The category $\C$ \defn{has enough normal-projectives} when for every object $Q$ of $\C$ there exists a normal-projective object $P$ and a normal epimorphism $p\colon {P\to Q}$. This $p$ is \defn{weakly universal} amongst normal epimorphisms with codomain $Q$, in the sense that for any normal epimorphism $q\colon X\to Q$, a morphism $x\colon P\to X$ exists such that $qx=p$.

Semi-abelian varieties of algebras do always have enough normal-projectives, because the free objects are projective with respect to the regular epimorphisms (=~surjective algebra morphisms), and all regular epimorphisms are normal. In the context of an abelian category, all epimorphisms are normal, so we regain the usual definition of a projective object.

A \defn{syzygy} of an object $Q$ is a short exact sequence
\[
    \xymatrix{0 \ar[r] & \Omega(Q) \ar[r]^-w & P \ar[r]^-p & Q \ar[r] & 0}
\]
where the middle object $P$ is normal-projective. We sometimes refer to just the object $\Omega(Q)$ as a syzygy of $Q$. We then write $\Omega(Q)=\Omega^1(Q)$ for any chosen syzygy, and recursively define $\Omega^{n+1}(Q)\coloneq\Omega(\Omega^n(Q))$.

\begin{theorem}[Syzygy Theorem]\label{Syzygy Theorem}
    In a category with kernels and cokernels, with pullback-stable normal epimorphisms and with enough normal-projectives, for each $n\geq1$ and any objects $Q$ and $K$, there are surjections
    \[
        \begin{tikzcd}[cramped, sep=small]
            \Ext^1(\Omega^n(Q),K)\ar[r,epi]&\cdots \ar[r,epi] & \Ext^n(\Omega(Q),K) \ar[r,epi]& \Ext^{n+1}(Q,K)\text{.}
        \end{tikzcd}
    \]
\end{theorem}
\begin{proof}
    We explain why there is a surjection
    \[
        \begin{tikzcd}[cramped, sep=small]
            w\colon\Ext^n(\Omega(Q),K) \ar[r,epi]& \Ext^{n+1}(Q,K)
        \end{tikzcd}
    \]
    for any natural number $n\geq 1$. The rest of the claim then follows by induction. Note that here we are not assuming that these conglomerates are small.

    The function $w$ takes an exact sequence
    \[
        0\to K\to X_n\to\cdots\to X_1\to \Omega(Q)\to 0
    \]
    of length $n$ and splices it on top of the syzygy
    \[
        0\to \Omega(Q)\to P\to Q\to 0
    \]
    so that we obtain the exact sequence
    \[
        0\to K\to X_n\to\cdots\to X_1\to P\to Q\to 0\text{.}
    \]
    This function is well defined: it is obvious that it preserves the equivalence relation, since any zigzag between exact sequences of length $n$ induces a zigzag between exact sequences of length $n+1$.

    We show that $w$ is surjective: it is here that we use that normal epimorphisms are pullback-stable. Indeed, then any exact sequence of length $n+1$ from $K$ to $Q$ pulls back to an exact sequence of length $n$ from $K$ to $\Omega(Q)$ as in Figure~\ref{Figure Pullback}.
    \begin{figure}
        \resizebox{\textwidth}{!}{
        \xymatrix{&&&&&&&0 \ar[r] &\Omega(Q) \ar@/^1.5ex/[rd]_-{w} \ar@{.>}[dddd]^-{\overline{\alpha}_{1}} \ar@{.>}[r] & 0\\
        0 \ar@{.>}[r] & K \ar@{:}[dd] \ar@{.>}[r] &
        Y_{n} \ar@{:}[dd]_-{\alpha_n} \ar@{.>}[r]^-{g_{n}} &
        Y_{n-1} \ar@{:}[dd]_-{\alpha_{n-1}} \ar@{.>}[r] & \cdots \ar@{.>}[r] & Y_2 \ar@{:}[dd]_-{\alpha_3} \ar@{.>}[rr]^{g_2} \ar@{.>}@/_1.5ex/[rd]^-{e_{3}} &&
        Y_{1} \ar@{.>}[dd]^-{\alpha_{2}} \ar@{.>}@/^1.5ex/[ru]_-{g_{1}} &&
        P \ar[r]^-{p} \ar@{.>}[dd]^-{\alpha_{1}} &
        Q \ar@{=}[dd] \ar[r] & 0\\
        &&&&&& I_3 \ar@/^1.5ex/[rd]_-{m_3} \ar@{.>}@/_1.5ex/[ru]^-{l_{2}} &&&&\\
        0 \ar[r] & K \ar[r] &
        X_{n+1}  \ar[r]_-{f_{n+1}} &
        X_{n}  \ar[r] & \cdots \ar[r] & X_3 \ar@/^1.5ex/[ru]_-{e_3} \ar[rr]_{f_3} &&
        X_{2} \ar[rr]|\hole^(.75){f_{2}} \ar@/_1.5ex/[rd]_-{e_{2}} &&
        X_{1} \ar[r]_-{f_1}  &
        Q \ar[r] & 0\\
        &&&&&&&&I_{2} \ar@/_1.5ex/[ru]_-{m_{2}}}}
        \caption{Syzygy and pullback}\label{Figure Pullback}
    \end{figure}
    This then proves that $w$ is surjective, because when this sequence is spliced on top of the chosen syzygy, the result is connected to the given sequence.

    We explain the details. Since $p$ is a weakly universal normal epimorphism, we may choose $\alpha_1\colon P\to X_1$ so that $f_1\alpha_1=p$; restricting it to the kernels $w$ of $p$ and $m_2$ of~$f_1$, we find $\overline{\alpha}_1\colon \Omega(Q)\to I_2$ such that $m_2\overline{\alpha}_1=\alpha_1w$. The construction of an exact sequence of length $n$ from $K$ to $\Omega(Q)$ starts here and depends on this choice of a morphism~$\overline{\alpha}_1$.

    Pull back the normal epimorphism $e_2$ to a normal epimorphism $g_1$ as in the figure. This also provides us with a morphism $\alpha_2\colon Y_1\to X_2$. By construction, the commutative square $\overline{\alpha}_1g_1=e_2\alpha_2$ is a pullback, so that the kernels of $e_2$ and $g_1$ coincide. More precisely, the kernel $m_3\colon I_3\to X_2$ of $e_2$ lifts over $\alpha_2$ to a morphism $l_2\colon I_3\to Y_1$ where $g_1l_2=0$ and which is easily seen to be a kernel of $g_1$. Because of this, we can put $Y_i=X_{i+1}$ and $g_{i+1}=f_{i+2}$ for all $i\geq 2$ and $g_2=l_2e_3\colon Y_2\to Y_1$.
\end{proof}

Together with Theorem~\ref{set many SES}, this gives us our main result:

\begin{theorem}\label{as many n exact sequences as SES}
    If, in a category with kernels and cokernels, with pullback-stable normal epimorphisms and with enough normal-projectives, $\Ext^1(Q,K)$ is small for all objects $Q$ and $K$, then so is $\Ext^n(Q,K)$ for each $n\geq1$. This is the case, for instance, in any semi-abelian variety of algebras.
\end{theorem}
\begin{proof}
    It suffices to prove that $\Ext^{n+1}(Q,K)$ is small if so is $\Ext^n(\Omega(Q),K)$. Since the $w$ in the proof of Theorem~\ref{Syzygy Theorem} is a surjection, its fibres form a partition of the conglomerate $\Ext^n(\Omega(Q),K)$. That conglomerate is bijective to a set, to which the given partition is transported. Any choice of representatives for this latter partition---which must exist under the Axiom of Choice---forms a set which is bijective to the conglomerate $\Ext^{n+1}(Q,K)$.
\end{proof}

\begin{remark}
    The dual procedure, involving normal-injectives, gives the same result under dual conditions. Our use of projectives instead has to do with the properties of the examples we study: projectives are more natural in the context of varieties of algebras, since the existence of enough normal-projectives comes for free in all semi-abelian varieties, and while normal epimorphisms are pullback-stable here, normal monomorphisms are not pushout-stable in general.
\end{remark}

\begin{remark}
    Theorem~\ref{as many n exact sequences as SES} implies, for instance, that for each object $K$ of a semi-abelian variety $\V$ and every $n\geq 1$, we have a functor $\Ext^n(-,K)\colon \V^\op\to \Set$. Here, for any morphism $\eta\colon Q\to Q'$, a function
    \[
        \Ext^n(\eta,K)\colon \Ext^n(Q',K) \to \Ext^n(Q,K)
    \]
    is obtained by pulling back exact sequences along $\eta$. The contravariance explains why we let $Q$ be the first variable in $\Ext^n(Q,K)$. Functoriality in the second variable $K$ is not automatic.
\end{remark}

\begin{remark}\label{Relative to E}
    The argument that make Theorem~\ref{Syzygy Theorem} work, can be extended to any pullback-stable class of normal epimorphisms $\E$, provided a weakly universal normal epimorphism $p\colon P\to Q$ exists in $\E$ for each object $Q$. Its kernel $\Omega_\E(Q)$ appears in the surjection
    \[
        \begin{tikzcd}[cramped, sep=small]
            w_\E\colon\Ext^n_\E(\Omega_\E(Q),K) \ar[r,epi]& \Ext^{n+1}_\E(Q,K)
        \end{tikzcd}
    \]
    existing for all $n\geq 1$. Here $\Ext^{i}_\E(Q,K)$ is defined as $\Ext^{i}(Q,K)$, but for exact sequences whose cokernel parts are in $\E$.
\end{remark}

\section{A variation on the theme:\texorpdfstring{\\ }{} double extensions and crossed extensions}\label{Section Variation}

We adapt the theory developed above to a slightly different situation: \emph{double} extensions and \emph{crossed} extensions instead of \emph{two-step} extensions. This section is not entirely self-contained; for the sake of a more compact presentation, we take definitions and results from the literature for granted, rather than explaining them in full detail.

In the article~\cite{RVdL2}, the Barr--Beck derived functors~\cite{Barr-Beck} of $\Hom(-,A)\colon \V^{\op}\to \Ab$, where $\Ab$ is the category of abelian groups and $A$ is an abelian object in a semi-abelian variety $\V$, are characterised in terms of so-called \emph{higher central extensions}. This generalises the well-known classification of central extensions via cohomology, hinted at in Example~\ref{Example Central Extensions}, to higher cohomology degrees. At the same time, it extends Yoneda's theory to a non-abelian setting---see~\cite{GPeschkeTVanderLinden2016} for more on this point of view.

We here consider the case of \emph{double} central extensions, which are equivalence classes of particular \emph{$(3\times 3)$-diagrams} such as occur for instance in the statement of the classical $(3\times 3)$-Lemma: see Figure~\ref{Figure 3x3}.
\begin{figure}
    $\vcenter{\xymatrix@!0@=3.5em{& 0 \ar@{->}[d] & 0 \ar@{->}[d] & 0 \ar@{->}[d]\\
        0 \ar@{->}[r] & K \ar[r] \ar[d] & X_2' \ar[d] \ar[r] & I_2' \ar[d] \ar@{->}[r] & 0\\
        0 \ar@{->}[r] & X_2 \ar[r] \ar[d] & Y \ar[r] \ar[d] & X_1' \ar[d] \ar@{->}[r] & 0\\
        0 \ar@{->}[r] & I_2 \ar[r] \ar[d] & X_1 \ar[r] \ar@{->}[d] & Q \ar@{->}[d] \ar@{->}[r] & 0\\
        & 0 & 0 & 0 }}$
    \caption{A $(3\times 3)$-diagram: rows and columns are short exact sequences}\label{Figure 3x3}
\end{figure}
Morphisms of such diagrams are the obvious natural transformations. Just like for $n$-step extensions, zigzags of those morphisms that keep the endpoints $K$ and~$Q$ fixed determine equivalence classes of $(3\times 3)$-diagrams. These are called \defn{double extensions under $K$ and over~$Q$} or \defn{from $K$ to $Q$} and form a conglomerate which we here denote $\DExt(Q,K)$. A~double extension is \defn{central} when its representing $(3\times 3)$-diagrams satisfy a further condition which in the present context of a semi-abelian variety may be expressed as a commutator condition---see~\cite{EGVdL,RVdL,RVdL2,RVdL3} for further details.

Again we may ask the question, whether the conglomerate $\DExt(Q,K)$ is small: the answer is yes---see Theorem~\ref{Theorem 2-Ext Small} below---by a variation on Theorem~\ref{Syzygy Theorem}. Note that when we restrict to the subconglomerate $\DCExt(Q,K)$ determined by the double \emph{central} extensions, then this is known, and follows from the interpretation of this conglomerate in terms of derived functors of~\cite{RVdL2}, as soon as the category~$\V$ satisfies an additional commutator condition, called the \emph{Smith is Huq} condition in~\cite{MFVdL}. A~straightforward adaptation of the proof of Theorem~\ref{Syzygy Theorem} gives us:

\begin{theorem}\label{Theorem 2-Ext Small}
    In a semi-abelian variety $\V$, the conglomerate $\DExt(Q,K)$ of all double extensions between any two objects $K$ and $Q$ is small.
\end{theorem}
\begin{proof}[Proof sketch]
    A $(3\times 3)$-diagram as in Figure~\ref{Figure 3x3} is completely determined by two types of data: the pullback square on the left below, and the induced short exact sequence on the right.
    \[
        \vcenter{\xymatrix@!0@=5em{X_1\times_QX_1' \ar[r] \ar[d] & X_1' \ar[d] \\
        X_1 \ar[r]  & Q }}
        \qquad\qquad
        \vcenter{\xymatrix{0 \ar[r] & K \ar[r] & Y \ar[r] & X_1\times_Q X_1' \ar[r] & 0}}
    \]
    The main reason for this is, that the bottom right square in Figure~\ref{Figure 3x3} is a so-called \emph{regular pushout}~\cite{Bourn2003,Carboni-Kelly-Pedicchio}, which means that the universal comparison $Y\to X_1\times_Q X_1'$ is a regular epimorphism.

    Once and for all, we fix two weakly universal normal epimorphisms with codomain~$Q$, and use their pullback
    \[
        \vcenter{\xymatrix@!0@=5em{P\times_QP' \ar[r] \ar[d] & P' \ar[d] \\
        P\ar[r]  & Q }}
    \]
    as a kind of ``double syzygy'' of $Q$. Mimicking the proof of Theorem~\ref{Syzygy Theorem}, we may see that for any $(3\times 3)$-diagram as in Figure~\ref{Figure 3x3}, there is a morphism from this ``double syzygy'' to the pullback associated with the $(3\times 3)$-diagram. Then the associated short exact sequence pulls back to a short exact sequence from $K$ to $P\times_QP'$. This eventually leads to a proof that, since the conglomerate $\Ext^1(P\times_QP',K)$ is small, so is $\DExt(Q,K)$.
\end{proof}

An immediate consequence of this is, that the conglomerate $\DCExt(Q,K)$ is small, independently of any additional conditions on the semi-abelian variety $\V$. This solves a problem remarked upon in~\cite[Remark~4.2]{RVdL}. Via the classification (obtained in Theorem~5.3 of~\cite{RVdL}) of double central extensions in terms of the second cohomology group defined in~\cite{Rodelo:Thesis,Rodelo-Directions}, this in turn implies that that latter cohomology group is small as well.

Restricting ourselves to the context of Moore categories~\cite{Rodelo:Moore}, the analysis of~\cite{Rodelo:Thesis,Rodelo-Directions} explains that the \emph{two-fold crossed extensions} in any strongly semi-abelian variety form a small set. Let us explain what this means in the special case of the variety of groups.

Recall that a \defn{crossed extension (of groups)}~\cite{Holt,Huebschmann} is an exact sequence of groups of length $2$
\[
    \xymatrix{0 \ar[r] & K \ar[r]^k & X_2 \ar[r]^-{f_2} & X_1 \ar[r]^q & Q \ar[r] & 0}
\]
together with an action of $X_1$ on $X_2$ making $f$ into a crossed module. Morphisms of crossed extensions are morphisms of exact sequences, compatible with the action. From the above, it follows that zigzags of those morphisms that keep the endpoints $K$ and $Q$ fixed determine equivalence classes which from a small conglomerate. This argument is actually valid in all strongly semi-abelian varieties (where now crossed modules are as defined in~\cite{Janelidze} and actions correspond to split short exact sequences by~\cite{Bourn-Janelidze}).

Using the ideas of Section~\ref{Section Syzygy}, these results generalise to $n$-fold crossed extensions---as in~\cite{Rodelo:Thesis,Rodelo-Directions}, extending the definitions of~\cite{Holt,Huebschmann} to Moore categories---for arbitrary $n\geq 2$. Indeed, a simple adaptation of the proof of Theorem~\ref{Theorem 2-Ext Small} to $n$-cubes makes it work for \emph{$n$-fold} extensions (in the sense of~\cite{RVdL2}) of arbitrary degree. Rather than working this out in detail, we chose, however, to end our article with an example of a slightly different nature.

\section{Schreier extensions of monoids}\label{Monoids}
The category of monoids is not semi-abelian, so the above does not apply as such. Nevertheless we may obtain a result, similar to Theorem~\ref{set many SES}, if only we suitably restrict the collection of short exact sequences we consider.

By definition~\cite{MR52404}, a \defn{Schreier extension} of monoids is a pair $(k,q)$ of monoid morphisms
\[
    \xymatrix{K \ar[r]^k & X \ar[r]^q & Q}
\]
where $q$ is a surjection and $k$ is a kernel of $q$, such that for each $v\in Q$ there exists an $x_v\in q^{-1}(v)$ such that for each $x\in q^{-1}(v)$ there is a unique $u\in K$ satisfying ${x=k(u)+x_v}$. Note that here the monoid operation is denoted additively. We may, and will, always assume that $x_0=0\in X$.

As explained in~\cite{MartinsFerreiraMontoliPatchkoriaSobral}, it is easy to see that then $q$ is a cokernel of~$k$, so that a Schreier extension is always a short exact sequence. In other words, ``being a Schreier extension'' is a property a short exact sequences of monoids may or may not satisfy. We let $\S$ be the class of normal epimorphisms of monoids underlying a Schreier extension.

We follow the slightly alternative view of~\cite{NMF-Semibiproducts} (which was later published as~\cite{MR4560197}), where it is stated that a Schreier extension may be defined as a short exact sequence $(k,q)$ of monoids as in
\[
    \xymatrix{0 \ar[r] & K \ar@<.5ex>[r]^-k & X \ar@{.>}@<.5ex>[l]^-p \ar@<.5ex>[r]^-q & Q \ar@{.>}@<.5ex>[l]^-s \ar[r] & 0}
\]
for which there exist functions $s\colon Q\to X$ and $p\colon X\to K$ such that
\begin{align*}
    qs           & =1_Q                                      \\
    kp+sq        & =1_X                                      \\
    p(k(u)+s(v)) & =u, \qquad \text{for $u\in K$, $v\in Q$.}
\end{align*}
Furthermore, $s$ may be chosen in such a way that $s(0)=0$.

The two viewpoints are indeed equivalent; let us explain one of the implications. For the definition of $s$, simply choose $s(v)$ amongst those $x_v$ that satisfy the requirements of the definition. Then $p$ sends $x\in X$ to the unique $u\in K$ for which ${x=k(u)+x_v}$. The conditions on $s$ and $p$ are readily verified.

Schreier extensions being special short exact sequences, for each $n\geq 1$, we may restrict the construction of the conglomerate $\Ext^n(Q,K)$ to Schreier exact sequences, as explained in Remark~\ref{Relative to E}, and thus define a conglomerate ${\Ext}_\S^n(Q,K)$. The question then arises, whether these conglomerates are small. When $n=1$, the answer is in essence the same as for extensions in semi-abelian varieties, since we may prove the following variation on Corollary~\ref{Tim retract}:

\begin{proposition}\label{Schreier retract}
    Any Schreier extension
    \[
        \xymatrix{0 \ar[r] & K \ar[r]^k & X \ar[r]^q & Q \ar[r] & 0}
    \]
    is a retract over $Q$ of the short exact sequence of monoids
    \[
        \xymatrix@C=3em{0 \ar[r] & K \ar[r]^-{(1_{K},0)} & K\times Q \ar[r]^-{\pi_Q} & Q \ar[r] & 0}
    \]
    in the category of pointed sets.
\end{proposition}
\begin{proof}
    We treat the kernel $k$ as a submonoid inclusion, which we may do without any loss of generality. We choose functions $s$ and $p$ as in the definition of a Schreier extension and define functions
    \begin{align*}
        \phi\colon K\times Q\to X & \colon (u,v)\mapsto k(u)+s(v)       \\
        \psi\colon X\to K\times Q & \colon x\mapsto (p(x),q(x))\text{.}
    \end{align*}
    For any $x\in X$, we have $\phi(\psi(x))=\phi(p(x),q(x))=k(p(x))+s(q(x))=x$, so that $\phi\psi=1_X$. Furthermore, $\pi_Q\psi=q$, while $q\phi=\pi_Q$ because $q(\phi(u,v))=q(k(u)+s(v))=q(s(v))=v$. Finally, $\phi$ preserves $0$ because so do $k$ and $s$, and $\psi$ preserves $0$ because $q$ and $p$ do. In the case of $p$, it suffices to note that $0=k(0)+s(0)$.
\end{proof}

We may then essentially copy the proof of Theorem~\ref{set many SES} to find

\begin{theorem}\label{set many SES Monoids}
    For any monoids $K$, $Q$, the conglomerate ${\Ext}_\S^1(Q,K)$ is small.\noproof
\end{theorem}

In view of the results of Section~\ref{Section Syzygy}, with in particular Remark~\ref{Relative to E} and the fact (easy to check by hand) that Schreier extensions are stable under pulling back, an interesting question seems to be, whether for each monoid $Q$, a weakly universal Schreier extension exists. If so, then also all of the ${\Ext}_\S^n(Q,K)$ where $n\geq 1$ are small.

This issue is closely related to the question, whether Schreier extensions are reflective amongst surjective monoid morphisms. Note that, though it solves a similar problem, the result in~\cite{MRVdL2} does not imply this. More suitable is the article~\cite{BClarke}, where it is conjectured that this should indeed be the case for Schreier \emph{split} extensions and split epimorphisms of monoids. We hope to study this problem in subsequent work.

\section*{Acknowledgement}
Many thanks to Beppe Metere and George Peschke for bringing up the question which led to this article, and for the interesting discussions this led to. We are also very grateful to the referee for helping us clarify how our work relates to the many alternative approaches---especially those of a homotopical-algebraic nature---to closely related problems in the literature.


\begin{thebibliography}{10}

    \bibitem{SGA4}
    M.~Artin, A.~Grothendieck, and J.-L. Verdier, \emph{Th{\'e}orie des topos et cohomologie {\'e}tale des sch{\'e}mas. {T}ome 1: Th{\'e}orie des topos}, Lecture Notes in Math., vol. 269, Springer, 1972, S{\'e}minaire de G{\'e}ometrie Alg{\'e}brique du Bois-Marie 1963-1964 (SGA4).

    \bibitem{Barr-Beck}
    M.~Barr and J.~Beck, \emph{Homology and standard constructions}, Seminar on triples and categorical homology theory ({ETH}, {Z\"u}rich, 1966/67), Lecture Notes in Math., vol.~80, Springer, 1969, pp.~245--335.

    \bibitem{Barr-Grillet-vanOsdol}
    M.~Barr, P.~A. Grillet, and D.~H. van Osdol, \emph{Exact categories and categories of sheaves}, Lecture Notes in Math., vol. 236, Springer, 1971.

    \bibitem{Borceux-Bourn}
    F.~Borceux and D.~Bourn, \emph{Mal'cev, protomodular, homological and semi-abelian categories}, Math. Appl., vol. 566, Kluwer Acad. Publ., 2004.

    \bibitem{Borceux-Clementino}
    F.~Borceux and M.~M. Clementino, \emph{Topological semi-abelian algebras}, Adv.~Math. \textbf{130} (2005), 425--453.

    \bibitem{Bourn1991}
    D.~Bourn, \emph{Normalization equivalence, kernel equivalence and affine categories}, Category {T}heory, {P}roceedings {C}omo 1990 (A.~Carboni, M.~C. Pedicchio, and G.~Rosolini, eds.), Lecture Notes in Math., vol. 1488, Springer, 1991, pp.~43--62.

    \bibitem{Bourn2003}
    D.~Bourn, \emph{The denormalized {$3\times 3$} lemma}, J.~Pure Appl. Algebra \textbf{177} (2003), 113--129.

    \bibitem{Bourn-Janelidze}
    D.~Bourn and G.~Janelidze, \emph{Characterization of protomodular varieties of universal algebras}, Theory Appl. Categ. \textbf{11} (2003), no.~6, 143--147.

    \bibitem{Carboni-Kelly-Pedicchio}
    A.~Carboni, G.~M. Kelly, and M.~C. Pedicchio, \emph{Some remarks on {M}altsev and {G}oursat categories}, Appl. Categ. Structures \textbf{1} (1993), 385--421.

    \bibitem{BClarke}
    B.~Clarke, \emph{Lifting twisted coreflections against delta lenses}, Theory Appl. Categ. \textbf{41} (2024), no.~26, 838--877.

    \bibitem{MMClementinoAMontoliLSousa2015}
    M.~M. Clementino, A.~Montoli, and L.~Sousa, \emph{Semidirect products of (topological) semi-abelian algebras}, J. Pure Appl. Algebra \textbf{219} (2015), 183--197.

    \bibitem{CRVdL}
    M.~Culot, F.~Renaud, and T.~Van~der Linden, \emph{Non-additive derived functors}, Glasgow Math. J. (2025), 44 pages, to appear.

    \bibitem{MR2102294}
    W.~G. Dwyer, P.~S. Hirschhorn, D.~M. Kan, and J.~H. Smith, \emph{Homotopy limit functors on model categories and homotopical categories}, Mathematical Surveys and Monographs, vol. 113, American Mathematical Society, Providence, RI, 2004.

    \bibitem{EGVdL}
    T.~Everaert, M.~Gran, and T.~Van~der Linden, \emph{Higher {H}opf formulae for homology via {G}alois {T}heory}, Adv.~Math. \textbf{217} (2008), no.~5, 2231--2267.

    \bibitem{EverVdL4}
    T.~Everaert and T.~Van~der Linden, \emph{Galois theory and commutators}, Algebra Universalis \textbf{65} (2011), no.~2, 161--177.

    \bibitem{Gabriel-Zisman}
    P.~Gabriel and M.~Zisman, \emph{Calculus of fractions and homotopy theory}, Springer, 1967.

    \bibitem{MR3958087}
    M.~Gran, F.~Sterck, and J.~Vercruysse, \emph{A semi-abelian extension of a theorem by {T}akeuchi}, J. Pure Appl. Algebra \textbf{223} (2019), no.~10, 4171--4190.

    \bibitem{Gran-VdL}
    M.~Gran and T.~Van~der Linden, \emph{On the second cohomology group in semi-abelian categories}, J.~Pure Appl.\ Algebra \textbf{212} (2008), 636--651.

    \bibitem{JRAGrayNMartinsFerreira2018}
    J.~R.~A. Gray and N.~Martins-Ferreira, \emph{New exactness conditions involving split cubes in protomodular categories}, Theory Appl. Categ. \textbf{33} (2018), 1031--1058.

    \bibitem{HHerrlichGEStrecker1973}
    H.~Herrlich and G.~E. Strecker, \emph{{Category Theory}}, 3 ed., {Sigma Ser. Pure Math.}, vol.~1, {Heldermann Verlag}, 2007.

    \bibitem{Higgins}
    P.~J. Higgins, \emph{Groups with multiple operators}, Proc. Lond. Math. Soc. (3) \textbf{6} (1956), no.~3, 366--416.

    \bibitem{Holt}
    D.~Holt, \emph{An interpretation of the cohomology groups {$H^{n}(G,M)$}}, J.~Algebra \textbf{60} (1979), 307--318.

    \bibitem{Huebschmann}
    J.~Huebschmann, \emph{Crossed $n$-fold extensions of groups and cohomology}, Comment. Math. Helv. \textbf{55} (1980), 302--314.

    \bibitem{Inyangala}
    E.~Inyangala, \emph{Categorical semi-direct products in varieties of groups with multiple operators}, Ph.D. thesis, University of Cape Town, 2010.

    \bibitem{Janelidze}
    G.~Janelidze, \emph{Internal crossed modules}, Georgian Math. J. \textbf{10} (2003), no.~1, 99--114.

    \bibitem{Janelidze-Marki-Tholen}
    G.~Janelidze, L.~M{\'a}rki, and W.~Tholen, \emph{Semi-abelian categories}, J.~Pure Appl. Algebra \textbf{168} (2002), no.~2--3, 367--386.

    \bibitem{Johnstone:Heyting}
    P.~T. Johnstone, \emph{A note on the semiabelian variety of {H}eyting semilattices}, {G}alois Theory, {H}opf Algebras, and Semiabelian Categories, Fields Inst. Commun., vol.~43, Amer. Math. Soc., 2004, pp.~317--318.

    \bibitem{MacLane:Homology}
    S.~{Mac\,Lane}, \emph{Homology}, Grundlehren math. Wiss., vol. 114, Springer, 1967.

    \bibitem{NMF-Semibiproducts}
    N.~Martins-Ferreira, \emph{Semi-biproducts of monoids}, preprint {\texttt{arXiv:2109.06278}}, 2021.

    \bibitem{MR4560197}
    N.~Martins-Ferreira, \emph{Pointed semibiproducts of monoids}, Theory Appl. Categ. \textbf{39} (2023), no.~6, 172--185.

    \bibitem{MartinsFerreiraMontoliPatchkoriaSobral}
    N.~Martins-Ferreira, A.~Montoli, A.~Patchkoria, and M.~Sobral, \emph{On the classification of {S}chreier extensions of monoids with non-abelian kernel}, Forum Math. \textbf{32} (2020), no.~3, 607--623.

    \bibitem{MFVdL}
    N.~Martins-Ferreira and T.~Van~der Linden, \emph{A note on the ``{S}mith is {H}uq'' condition}, Appl.\ Categ.\ Structures \textbf{20} (2012), no.~2, 175--187.

    \bibitem{MRVdL2}
    A.~Montoli, D.~Rodelo, and T.~Van~der Linden, \emph{On the reflectiveness of special homogeneous surjections of monoids}, Categorical Methods in Algebra and Topology, Textos de Matem{\'a}tica (S{\'e}rie~B), vol.~46, Departamento de Matem{\'a}tica da Universidade de Coimbra, 2014, pp.~237--244.

    \bibitem{GPeschkeTVanderLinden2016}
    G.~Peschke and T.~Van~der Linden, \emph{{The Yoneda isomorphism commutes with homology}}, J. Pure Appl. Algebra \textbf{220} (2016), 495--517.

    \bibitem{PVdL2}
    G.~Peschke and T.~Van~der Linden, \emph{A homological view of categorical algebra}, Preprint \texttt{arXiv:2404.15896}, 2024.

    \bibitem{Quillen}
    D.~G. Quillen, \emph{Homotopical algebra}, Lecture Notes in Math., vol.~43, Springer, 1967.

    \bibitem{MR52404}
    L.~R\'edei, \emph{Die {V}erallgemeinerung der {S}chreierschen {E}rweiterungstheorie}, Acta Sci. Math. (Szeged) \textbf{14} (1952), 252--273.

    \bibitem{Rodelo:Moore}
    D.~Rodelo, \emph{Moore categories}, Theory Appl. Categ. \textbf{12} (2004), no.~6, 237--247.

    \bibitem{Rodelo:Thesis}
    D.~Rodelo, \emph{Direc{\c c\~ o}es para a sucess{\~a}o longa de co-homologia}, Ph.D. thesis, Universidade de Coimbra, 2005.

    \bibitem{Rodelo-Directions}
    D.~Rodelo, \emph{Directions for the long exact cohomology sequence in {M}oore categories}, Appl. Categ. Structures \textbf{17} (2009), no.~4, 387--418.

    \bibitem{RVdL}
    D.~Rodelo and T.~Van~der Linden, \emph{The third cohomology group classifies double central extensions}, Theory Appl. Categ. \textbf{23} (2010), no.~8, 150--169.

    \bibitem{RVdL3}
    D.~Rodelo and T.~Van~der Linden, \emph{Higher central extensions via commutators}, Theory Appl. Categ. \textbf{27} (2012), no.~9, 189--209.

    \bibitem{RVdL2}
    D.~Rodelo and T.~Van~der Linden, \emph{Higher central extensions and cohomology}, Adv. Math. \textbf{287} (2016), 31--108.

    \bibitem{VdLinden:Simp}
    T.~Van~der Linden, \emph{Simplicial homotopy in semi-abelian categories}, J.~K-Theory \textbf{4} (2009), no.~2, 379--390.

    \bibitem{Verdier}
    J.-L. Verdier, \emph{Des cat{\'e}gories d{\'e}riv{\'e}es des cat{\'e}gories ab{\'e}liennes}, Ph.D. thesis, Universit{\'e} de Paris, 1967.

    \bibitem{MR1453167}
    J.-L. Verdier, \emph{Des cat\'egories d\'eriv\'ees des cat\'egories ab\'eliennes}, Ast\'erisque (1996), no.~239, xii+253, With a preface by Luc Illusie, Edited and with a note by Georges Maltsiniotis.

    \bibitem{Weibel}
    Ch.~A. Weibel, \emph{An introduction to homological algebra}, Cambridge Stud. Adv. Math., vol.~38, Cambridge Univ. Press, 1994.

    \bibitem{Yoneda-Exact-Sequences}
    N.~Yoneda, \emph{On {E}xt and exact sequences}, J. Fac. Sci. Univ. Tokyo \textbf{1} (1960), no.~8, 507--576.

\end{thebibliography}

\end{document}